\documentclass[12pt]{amsart}
\usepackage{bbm}
\usepackage{amscd, amssymb, dsfont, upgreek, color}
\input{xy}
\xyoption{all}

\newtheorem{Thm}{Theorem}[section]
\newtheorem{Conj}[Thm]{Conjecture}
\newtheorem{Prop}[Thm]{Proposition}
\newtheorem{Def}[Thm]{Definition}
\newtheorem{Def/Thm}[Thm]{Definition/Theorem}

\newtheorem{Lemma}[Thm]{Lemma}

\theoremstyle{remark}
\newtheorem{Rmk}[Thm]{Remark}

\numberwithin{equation}{subsection}

\newcommand{\ot }{\otimes}
\newcommand{\ra }{\rightarrow}

\newcommand{\lra }{\longrightarrow}

\newcommand{\Hom }{{\mathrm{Hom}}}

\newcommand{\Spec}{{\mathrm{Spec}}}

\newcommand{\Pic}{{\mathrm{Pic}}}
\newcommand{\Sym}{{\mathrm{Sym}}}

\newcommand{\cP}{{\mathcal{P}}}

\newcommand{\cC}{{\mathcal{C}}}

\newcommand{\G}{{\bf G}}

\newcommand{\bT}{{\bf T}}

\newcommand{\PP }{{\mathbb P}}

\newcommand{\QQ }{{\mathbb Q}}
\newcommand{\CC }{{\mathbb C}}
\newcommand{\ZZ }{{\mathbb Z}}

\newcommand{\one }{{\mathbbm 1}}

\newcommand{\ke }{{\varepsilon }}

\newcommand{\kg }{{\gamma}}

\newcommand{\kl }{{\lambda}}

\newcommand{\QGraphok}{QG_{0,0|k,\beta}^{\theta, \ke} ([W/\G])}

\newcommand{\vir}{\mathrm{vir}}

\newcommand{\aff}{\mathrm{\!\!aff}}
\newcommand{\nov}{\mathrm{nov}}

\newcommand{\eb}{ev_\bullet}

\newcommand{\WmodG}{W/\!\!/\G}
\newcommand{\WmodtG}{W/\!\!/_{\!\theta}\G}

\newcommand{\T}{{\bf T}}

\newcommand{\lan}{\langle}
\newcommand{\ran}{\rangle}
\newcommand{\lla}{\langle\!\langle}
\newcommand{\rra}{\rangle\!\rangle}

\newcommand{\id}{\mathrm{id}}
\newcommand{\dsJ}{\mathds{J}}
\newcommand{\dsS}{\mathds{S}}
\newcommand{\dsI}{\mathds{I}}
\newcommand{\uE}{\underline{E}}

\newcommand{\uke}{\upepsilon}

\newcommand{\uev}{\underline{ev}}
\newcommand{\upi}{\underline{\pi}}

\begin{document}

\title{Big $I$-functions} 
\dedicatory {To Professor Shigeru Mukai, on the occasion of his $60^{th}$ birthday.}
\begin{abstract} 
We introduce a new big $I$-function for  certain GIT quotients $\WmodG$ using the quasimap graph space 
from infinitesimally pointed $\PP ^1$ to the stack quotient $[W/\G]$.
This big $I$-function is expressible by the small $I$-function introduced in \cite{CK, CKM}.
The $I$-function conjecturally generates the Lagrangian cone of 
Gromov-Witten theory for $\WmodG$ defined by Givental. We prove the conjecture when
$\WmodG$ has a torus action with good properties.
\end{abstract}

\author{Ionu\c t Ciocan-Fontanine}
\noindent\address{School of Mathematics, University of Minnesota, 206 Church St. SE,
Minneapolis MN, 55455, and\hfill
\newline \indent School of Mathematics, Korea Institute for Advanced Study,
85 Hoegiro, Dongdaemun-gu, Seoul, 02455, Korea}
\email{ciocan@math.umn.edu}

\author{Bumsig Kim}
\address{School of Mathematics, Korea Institute for Advanced Study,
85 Hoegiro, Dongdaemun-gu, Seoul, 02455, Korea}
\email{bumsig@kias.re.kr}

\maketitle


\section{Introduction}

Let $X$ be a nonsingular quasi-projective variety with a torus $\T$-action such that the $\T$-fixed
locus $X^{\T}$ is projective. We allow $\T$  to be the trivial group. The $\T$-equivariant rational Gromov-Witten theory for $X$
is encoded in the genus $0$ prepotential $F$, i.e., the generating function of gravitational Gromov-Witten invariants
defined by the integration of psi-classes and pullbacks of cohomology classes of target $X$
against the virtual fundamental classes of the moduli space of $k$ pointed, genus $0$,  numerical class $\beta$ stable maps to $X$.

Givental shows that the graph of the formal $1$-form $dF$ is a Lagrangian cone in a suitably defined 
infinite dimensional symplectic space and the cone is generated by the $J$-function (see \cite{Giv_Sym}). The big $J$-function for $X$ is 
a generating function of genus $0$ GW-invariants with gravitational insertions at one point, and any number of primary insertions. It is a difficult problem to compute the $J$-function
in general. 
In the case when $X$ has a GIT presentation $X=\WmodG$ with $W$ affine, 
there is a replacement of the $J$-function. It is the so-called $I$-function, introduced in \cite{CK, CKM} as a generalization of Givental's small $I$-function for toric targets.  
While it is shown in \cite{CKg0} that $I$ and $J$ are related via generalized Mirror Theorems,
the big $I$-function is equally difficult to compute. 
The purpose of this paper is to remedy this situation by introducing a new version of $I$-functions (for the same kind of GIT targets). 
This new function, which we denote by $\dsI$, can be computed explicitly in closed form in many cases, and 
the $J$-function is obtained from it via the Birkhoff factorization procedure, as given in \cite{CG}.

The precise GIT set-up is as follows. Let $W$ be an affine variety with a linear right action of a reductive algebraic group $\G$.
For any rational character $\theta$ of $\G$, denote by 
 $W^{ss}(\theta)$ the semistable locus of $W$ with respect to $\theta$.
Assume that $W^{ss}(\theta )$ is nonsingular, 
$W$ has at worst l.c.i singularities, and $\G$ acts on $W^{ss}(\theta)$ freely (however, see \cite{CCK} for allowing finite
non-trivial stabilizers).

Given such a triple  $(W, \G, \theta)$, there is a relative compactification of the space of maps from $\PP ^1$ to $\WmodG$ 
of given numerical class $\beta$ (see Definition \ref{numerical class} for the notion of numerical class),
keeping the domain curve $\PP ^1$ but allowing maps $\PP ^1\ra [W/\G]$ to the stack quotient. The \lq\lq compactification"
is called the quasimap graph space and defined to be
\[ QG_{0,0,\beta}(\WmodG) := \{ f \in \mathrm{Hom}(\PP ^1, [W/\G]) :  f^{-1}(\WmodG) \ne \emptyset, \beta_f = \beta \}. \]
It is an algebraic space proper over the affine quotient $W/_{\aff}\G$ (see \cite{CKM}). This graph space
is equipped with a $\CC ^*$-action induced from the $\CC ^*$-action on $\PP ^1$, as well as with
a natural equivariant perfect obstruction theory.
There is a distinguished open and closed subspace $F_\beta$ of the $\CC^*$-fixed locus of the graph space
$QG_{0,0,\beta}(\WmodG)$. The small $I$-function is defined by the localization residue at $F_\beta$ as follows:
\[ \dsI _{sm}(q, z) := \sum _{\beta} q^\beta (ev_\bullet )_*( \mathrm{Res}_{F_{\beta}} [QG_{0,0,\beta}(\WmodG)]^{\vir}), \] 
where $ev_\bullet$ is the evaluation map from $F_\beta$ to $\WmodG$ at the generic point of $\PP ^1$
and $z$ is the $\CC ^*$-equivariant parameter. The sum is over all $\theta$-effective \lq\lq curve classes" $\beta\in \mathrm{Eff}(W, \G, \theta )$, 
see Definition \ref{effective class} for the notion of
$\theta$-effective class.

There is another evaluation map $\hat{ev}_\beta$ from $F_\beta$
at $0\in \PP ^1$. The codomain of $\hat{ev}_\beta$ is the stack quotient $[W/\G]$.
Therefore we have 
\[\xymatrix {  [W/\G]  & \ar[l]_{\ \ \ \hat{ev}_\beta}  F_{\beta} \ar[r]^{ev_\bullet \ \ } &   \WmodG }. \]
 The big $I$-function in this paper is
\[ \dsI ({\bf t} ) = \sum _\beta q^\beta  (ev_{\bullet})_*(\exp (\hat{ev}^*_\beta ({\bf t})/z ) \cap  \mathrm{Res}_{F_{\beta}} [QG_{0,0,\beta}(\WmodG)]^{\vir}   ),\]   
for $   {\bf t} \in H^*([W/\G],\QQ)$.

We conjecture that $\dsI ({\bf t})$ is on the Lagrangian cone of Gromov-Witten theory of $\WmodG$ with 
Novikov variables from $\mathrm{Eff}(W, \G, \theta )$.
We prove the conjecture when there is an action by a torus $\T$
on $W$, commuting with the $\G$-action, and such that $X=\WmodG$ has only isolated $0$ and $1$-dimensional $\T$-orbits. 

To prove the conjecture, we introduce the stable quasimaps with $\uke:=(1,..., 1, \ke, ..., \ke)$-weighted markings and 
the $J^\uke$-function whose special case is the $I$-function. The proof is parallel to the proof of the corresponding theorems  in
\cite{CKg0}.

In the last section we explain how to obtain an explicit closed formula for the big $\dsI ({\bf t})$ for toric varieties and for complete intersections in them.

\subsection{Acknowledgments}

The authors thank V. Alexeev, A. Bondal, T. Coates, D.E. Diaconescu, K. Hori, H. Iritani, M. Jinzenji, Y. Toda and H.-H. Tseng for inspiring discussions and comments. 
The authors thank also the referee for useful suggestions that helped to improve the presentation of the paper.
The research of the first named author was supported in part by the NSF grant DMS--1305004. 
The research of the second named author was partially supported by the KRF grant 2007-0093859.


\section{Weighted Stable Quasimaps}

Throughout the paper the base field is $\CC$.

\subsection{$\theta$-stable quasimaps}

Let $\chi (\G ):=\Hom (\G, \CC ^*)$ be the group of characters of a reductive algebraic group $\G$.
For $\theta\in \chi (\G )$  and  a positive rational number $\ke$, 
the notion of  $\ke$-stable quasimaps to the GIT quotient $\WmodtG=[W^{ss}(\theta)/\G]$\
was introduced in \cite{CKM} provided with the following assumption:

\medskip

\noindent{\em  Condition $\bigstar$}:  The $\G$-action on the semistable locus
$W^{ss}(\theta)$ with respect to $\theta$ is free. 

\medskip

Note that condition $\bigstar$ guarantees that the stable and semi-stable loci in $W$ for the linearization of the action given by $\theta$ coincide.

It will be convenient to extend the notion of stability to a {\it rational} character $\theta$, while removing $\ke$. This is based on the observation from \cite[Remark 7.1.4]{CKM}
that $\ke$-stability with respect to the integral character $\theta$ is equivalent to $\frac{\ke}{m}$-stability with respect to $m\theta$, for every positive integer $m$,
and is done
as follows. Let 
\[ \theta \in \chi (\G) _\QQ := \chi (\G )\ot _{\ZZ} \QQ \]  be  a rational character of $\G$.
Denote by $L_\theta$ the $\QQ$-line bundle
on $[W/\G]$ associated to the character $\theta$, namely,
 $$L_\theta := (W\times \CC _{m\theta})^{\ot 1/m},$$
for any positive integer $m$ making $m\theta$ integral,
where $\CC _{m\theta}$ stands for the 1-dimensional $\G$-representation space given by the character $m\theta$. Here and in the rest of the paper we identify as usual the $\G$-equivariant
Picard group of $W$ with the Picard group of the quotient stack $[W/\G]$.
The unstable closed subscheme $W^{un}(\theta)\subset W$ is defined as $W^{un}(m\theta)$, and the semistable locus is the open subscheme $W^{ss}(\theta):=W\setminus W^{un}(\theta)$.
The semistable locus is independent on the choice of $m\in\ZZ_{>0}$ with $m\theta\in\chi(\G)$. We require that $\theta$ satisfies
Condition $\bigstar$ (this makes sense by the above discussion), so that $\WmodtG=[W^{ss}(\theta)/\G]$.

\begin{Def}\label{numerical class} Let $C$ be a (possibly disconnected) reduced, projective, at worst nodal curve.
The {\em numerical class} of a morphism $f: C \ra [W/\G]$ is the homomorphism of abelian groups $$\beta_f \in \Hom (\Pic [W/\G], \ZZ )$$ 
given by $$\beta_f(L)= \deg f^* (L) $$ for $L\in \mathrm{Pic}([W/\G])$.
\end{Def}

\begin{Def}\label{prestable map}
Let $(C, {\bf x}):=(C, x_1, ..., x_k)$ be a genus $g$, $k$-pointed prestable curve over the field $\CC$.
 (Recall this means that $C$ is a reduced, projective, connected, at worst nodal curve of arithmetic genus $g$, and $x_i$ are distinct nonsingular closed points in $C$.)
A morphism $$f:C\lra [W/\G]$$
is called a {\em $k$-pointed prestable map of genus $g$} to $[W/\G]$.
\end{Def}

\begin{Def}\label{theta-quasimap} Let $((C, {\bf x}), f)$ be a prestable map to $[W/\G]$.
\begin{itemize}
\item The {\em base locus of $f$ with respect to $\theta$} is  $$f^{-1}([W^{un}(\theta )/\G ]):=[W^{un}(\theta )/\G ]\times_{[W/\G]} C$$
with the reduced scheme structure.

\item $((C, {\bf x}), f)$ is called a {\em $\theta$-quasimap to $[W^{ss}(\theta )/\G]$} if the base locus with respect to $\theta$
is $0$-dimensional.

\item A $\theta$-quasimap $((C, {\bf x}), f)$ is called {\em $\theta$-prestable}
 if the base locus is away from all nodes of $C$.\footnote{The definition of prestability given here differs slightly from that in \cite[Definition 3.1.2]{CKM}, 
 as we now allow base-points to occur at the markings of a prestable quasimap. The stability condition (2) in Definition \ref{Def1} below implies
that there are no base-points at markings for {\it stable} quasimaps. This choice of definitions is more natural from the perspective of the weighted case introduced in \S 2.2.}
\end{itemize}
\end{Def}

By \cite[Lemma 3.2.1]{CKM}, a $\theta$-quasimap satisfies $$\beta_f(L_\theta)\geq 0,$$
with equality if and only if $\beta_f=0$, if and only if $f$ is a constant map to the GIT quotient $W/\!\!/_\theta\G =[W^{ss}(\theta)/\G]$.

\begin{Def}\label{length}
Let $((C, {\bf x}), f)$ be a $\theta$-prestable quasimap to $[W^{ss}(\theta )/\G]$.
The {\em $\theta$-length} $\ell_\theta (p)$  of $f$ at a smooth closed point $p$ of $C$ is defined as follows:
Choose $\ke ' \in \QQ _{>0}$ such that $\theta ' = \frac{1}{\ke '} \theta \in \chi (\G)$ is an integral character. Then
\begin{equation*}  \ell_\theta (p) := \ke' \ell_{\theta '} (p),
\end{equation*}
where
$\ell_{\theta ' }(p)$ is the length defined in \cite[Definition 7.1.1]{CKM}.
\end{Def}

\begin{Rmk}\label{length remark} The following properties are immediate to check from the above definition:
\begin{enumerate}
\item $\ell_\theta(p)$ is a well-defined rational number (i.e., it does not depend on the choice of $\ke'$ and $\theta'$). If $\lambda\in\QQ_{>0}$, then $\ell_{\lambda\theta}(p)=\lambda\ell_\theta (p)$.
\item For every nonsingular point $p\in C$, 
$$0\leq \ell_\theta(p)\leq \beta_f(L_\theta)$$
and $\ell_\theta(p) >0$ if and only if $p$ is in the base locus of $f$.

\item  Suppose that $W$ is a product $W_1\times W_2$ of two affine varieties $W_i$
 with component-wise $\G := \G_1 \times \G_2$-action such that Condition $\bigstar$ holds for each pair $(W_i,\theta_i)$. Here $\theta _i$  is the 
 character of the reductive group $\G_i$ induced from the
 character $\theta$ of $\G$, so that $\theta = \theta _1 \oplus \theta _2$.
 For a prestable map $$f=(f_1, f_2): C\ra [W/\G]=[W_1/\G_1]\times _{\Spec \CC} [W_2/\G_2]$$ and a smooth point $p\in C$, 
\[ \ell_\theta (p) = \ell_{\theta _{1}}(p) + \ell_{\theta _2} (p). \] 
This follows from the K\"unneth formula.

\end{enumerate}
\end{Rmk}

\begin{Def}\label{Def1}
 A $\theta$-prestable  quasimap  $((C, {\bf x}), f)$ is {\em $\theta$-stable}
if: 
\begin{enumerate}
\item $\omega _C(\sum x_i) \ot f^*L_\theta$ is ample and

\item  for every smooth point $p\in C$,
\begin{equation*} \ell_\theta (p) + \sum _i  \delta _{x_i, p}   \le 1 \end{equation*} 
where  $\delta _{x_i, p} :=1$ if $x_i=p$;  $\delta _{x_i, p} :=0$ if $x_i\ne p$. 
\end{enumerate}
\end{Def}

Note that the stability condition $(2)$ in Definition \ref{Def1} requires that $\ell_\theta(x_i)=0$ for each marking $x_i$. By Remark \ref{length remark}$(2)$, this says that the base locus of a 
$\theta$-stable quasimap is away from the markings of $C$.

\begin{Prop} \label{properties}
Let $\theta=\ke'\theta'$ with $\ke'\in\QQ_{>0}$ and $\theta'$ integral. Then

$(i)$
A prestable map $((C, {\bf x}), f)$ to $[W/\G]$ is $\theta$-stable if and only if it is a $\ke '$-stable quasimap to $W/\!\!/_{\theta '}\G$, as 
defined in  \cite[Definition 7.1.3]{CKM}. 

$(ii)$
A prestable map $((C, {\bf x}), f)$ to $[W/\G]$ with $\beta_f (L_\theta )\le 1$ is $\theta$-stable if and only if it is a stable quasimap to $W/\!\!/_{\theta '}\G$, as defined
in \cite[Definition 3.1.2]{CKM} (or a $(0+)$-stable quasimap to $W/\!\!/_{\theta '}\G$, in the terminology of \cite[Remark 2.4.7(2)]{CKg0}). 

$(iii)$ Let $\theta _0$ be the minimal integral character in the half ray $\QQ _{>0}\theta$. We write $\theta _1 > \theta _2$ if $\theta _1 = \lambda _1\theta _0$
and $\theta _2 = \lambda _2 \theta _0$ with two positive rational numbers $\lambda _1 >\lambda _2$.

If $\theta >  \theta _0$ and $(g, k)\ne (0, 0)$ ($\theta > 2 \theta _0$ when $(g, k)=(0,0)$),  
a prestable map $((C, {\bf x}), f)$ to $[W/\G]$ is $\theta$-stable if and only if it is a stable map to the quasi-projective scheme $W/\!\!/_{\theta '}\G$. 

\end{Prop}
\begin{proof} Left to the reader, as all statements follow easily from the definitions. \end{proof}

\begin{Def}\label{effective class}
An element $\beta \in \Hom _{\ZZ} (\Pic ([W/\G], \ZZ )$ is called {\em $\theta$-effective} (or equivalently $L_{\theta '}$-effective as in \cite[Definition 3.2.2]{CKM})
 if it can be realized as a finite sum of classes of 
$\theta$-quasimaps. 
\end{Def}

The subset $\mathrm{Eff} (W, G, \theta )\subset  \Hom _{\ZZ} (\Pic ([W/\G], \ZZ )$ of $\theta$-effective classes is a semigroup with no nontrivial invertible elements, i.e.,
$\beta _1+ \beta _2 =0$ for $\beta _i\in \mathrm{Eff} (W, G, \theta )$
 implies that $\beta _1=\beta _2 =0$  (see \cite[Lemma 3.2.1]{CKM}).

For a $\theta$-effective class $\beta$,
we denote by $Q_{g, k}^{\theta} ([W/\G], \beta )$ the moduli stack  of genus $g$, $k$-pointed $\theta$-stable quasimaps to $[W/\G]$ with numerical class $\beta$. By Proposition \ref{properties}$(i)$,
\begin{equation}\label{rationalQ} Q_{g, k}^{\theta} ([W/\G], \beta ) = Q_{g, k}^{\ke '} (W/\!\!/_{\theta '}\G , \beta ),\end{equation} 
where the right-hand side is the stack from
\cite[Theorem 7.1.6]{CKM}. Hence $Q_{g, k}^{\theta} ([W/\G], \beta )$ is a DM-stack, proper over the affine quotient 
$$W/_{\aff}\G := \Spec (A(W) ^\G), $$ where $A(W)$ denotes the affine coordinate ring of $W$.
 These moduli stacks  carry canonical perfect obstruction theories (see \cite[\S4.4-4.5]{CKM}).

\begin{Def}
A prestable map $((C, {\bf x}), f)$ to $[W/\G]$ which is $\lambda \theta$-stable for every $0< \lambda << 1$ is
called {\em $(0+)\cdot \theta$-stable}.  This notion is equivalent to the notion of stable quasimaps with respect to
$\theta '$ defined in \cite[Definition 3.1.2]{CKM}, where $\theta'$ is any integral character in the half ray $\QQ _{>0}\theta$.
See also \cite[Remark 2.4.7(2)]{CKg0}, where the terminology $(0+)$-stable quasimaps to $W/\!\!/_{\theta '}\G$ was used for the same notion.
\end{Def}

Therefore we define the corresponding moduli stacks by
\begin{equation} \label{asymptotic}
Q_{g, k}^{(0+)\cdot \theta} ([W/\G], \beta )     : = Q_{g, k}^{0+}(W/\!\!/_{\theta '}\G, \beta ), \end{equation}
where for the right-hand side we used the notation from \cite[Remark 2.4.7(2)]{CKg0}. They are also DM-stacks, proper over the affine quotient, carrying canonical perfect obstruction theories.

We discuss next $\theta$-stability for the quasimap graph spaces of \cite[\S7.2]{CKM} and \cite[\S2.6]{CKg0}.

Let $N\geq 1$ be an integer and consider the standard scaling action of $\CC^*$ on $\CC^N$. For $n\in \ZZ$ we have the character $$n\mathrm{id}:\CC^*\lra\CC^* ,\ \  t \mapsto t^n.$$
There are identifications
$$\ZZ\stackrel{\sim}{\lra}\chi(\CC^*)\stackrel{\sim}{\lra}\Pic([\CC^N/\CC^*], \;\; n\mapsto n\mathrm{id}\mapsto L_{n\mathrm{id}}.$$
For each $\beta\in\Hom(\Pic([W/\G]),\ZZ)$, define an abelian group homomorphism  $(\beta, 1)\in \Hom( \Pic ([W/\G]\times [\CC^N/\CC ^*]),\ZZ)$ by
$$ (\beta, 1)  (L\boxtimes L_{n\mathrm{id}}) =  \beta (L) + n .$$
Now we define the $\theta$-{\em stable quasimap graph space}:
\begin{equation}\label{graph} QG_{g, k, \beta}^\theta ([W/\G])  := Q_{g, k}^{\theta \oplus 3\mathrm{id}}([W\times \CC ^2 /\G \times \CC ^*], (\beta, 1)),
                       \end{equation}
                      where $\theta \oplus 3\mathrm{id}$ is a rational character of $\G\times \CC ^*$. As before, we see that
                      \begin{equation} QG_{g, k, \beta}^\theta ([W/\G])= QG_{g, k, \beta}^{\ke '} (W/\!\!/_{\theta ' }\G),\end{equation}
                      where the right-hand side is the graph space of $\ke'$-stable quasimaps to the GIT quotient (in the notation from \cite[\S2.6]{CKg0}).
                      
Finally, we have the graph spaces for the $(0+)\cdot\theta$-stability condition:
\begin{equation}\label{graph asymptotic}
       QG_{g, k, \beta}^{(0+)\cdot \theta} ([W/\G])  := QG_{g, k, \beta }^{0+} (W/\!\!/_{\theta '}\G).
       \end{equation}                      
Again, the graph spaces \eqref{graph} and \eqref{graph asymptotic} are DM-stacks, proper over the affine quotient,
and  carry canonical perfect obstruction theories.


\subsection{Weighted stable quasimaps}\label{mixed}

In this section, we introduce the weighted pointed stable quasimaps.
The moduli spaces of weighted pointed stable maps to a (quasi)projective target are constructed and studied in \cite{AG, BM, H}.
Recently, in \cite{Janda}, Janda considered the moduli space of weighted pointed stable quotients and its applications.
Also recently,
in \cite{Jin-Shi}, Jinzenji and Shimizu studied a graph space-type quasimap compactification of the moduli space 
of  maps  from $\PP^1$ to $\PP ^n$ with some weighted markings and its applications to generalized mirror maps.

Let  \[ (\theta, \upepsilon) := (\theta , \ke _1, ..., \ke _k) \in \chi (\G )_{\QQ} \times (\QQ _{>0})^k \]
such that $\theta$ satisfies Condition $\bigstar$ and $\ke _i \le 1$, $i=1, ..., k$.

\begin{Def}\label{weighted stability} A pair $((C, x_1,..., x_k), f ) $ is called a {\em $(\theta, \upepsilon)$-stable quasimap with weighted markings} and numerical class $\beta$
if:
 \begin{enumerate} \item {\em ($\uke$-weighted prestable map to $[W/\G]$)}

\begin{enumerate}

\item $C$ is a genus $g$, prestable curve over the field $\CC$.

\item $x_i$ are smooth points on $C$ (not necessarily pairwise distinct), with \[ \sum_i \ke _i \delta _{x_i, p} \le 1\]  for every smooth point $p$ of $C$.

\item $f$ is a morphism from $C$ to $[W/\G]$.

\end{enumerate}

\item {\em ($\theta$-quasimap)} $f^{-1}([W^{un}(\theta) /\G])$ is $0$-dimensional. 

\item {\em ($\theta$-prestability)}  $f^{-1}(\WmodtG )$ contains all nodes of $C$.

\item\label{Stab} {\em ($(\theta, \uke)$-stability)}

\begin{enumerate}
\item The $\QQ$-line bundle  \[  \omega _C (\sum _{i=1}^k \ke _i x_i ) \ot f^*L_{\theta} \]  is ample.

\item For every smooth point $p\in C$, 
 \[  \ell_{\theta} (p)  +\sum _{i=1}^{k}  \ke _i \delta _{ x_i,  p}    \le 1.\]
\end{enumerate}

\item {\em (numerical class $\beta$)} $\beta_f=\beta$.

\end{enumerate}

\end{Def}

By treating each marking $x_i$ as an effective divisor of $C$, there is a natural 
correspondence
\begin{align*} &  \left\{     \begin{array}{r} (f:C \ra [W/\G]) ,  \text{ together with ordered smooth points } \\
                                              x_i \in C, i=1, ..., k  :  \text{ class } \beta                
          \end{array} \right\}  \\
& \leftrightarrow
\left\{ \begin{array}{r}    (\tilde{f}:=(f, \pi _1, ..., \pi _k): C \ra [W/\G] \times [\CC/ \CC ^*]^{k}) :  \\
                                         \pi_i \text{ are id-prestable quasimaps to } [\CC /\CC ^*], \text{ class } (\beta, 1, ..., 1)  
           \end{array}\right\}. \end{align*}
Consider the rational character 
\[ \uptheta := \theta \oplus \underbrace{\ke _1\mathrm{id} \oplus \cdots \oplus \ke _k \mathrm{id}}_k \in \chi (\G\times (\CC ^*)^k)_{\QQ} . \]
Then $\tilde{f}^*(L_\uptheta ) = f^*(L_\theta ) \ot \mathcal{O}_C (\sum \ke _i x_i )$ 
and $\ell_\uptheta (p) = \ell_\theta (p) + \sum \ke _i \delta _{x_i, p}$.
Therefore, the $(\theta, \uke)$-stability of $((C, x_1,..., x_k), f ) $ from Definition \ref{weighted stability} translates  via the above correspondence into $\uptheta$-stability of $\tilde{f}$,
and so 
the moduli stack of $(\theta, \uke)$-stable quasimaps of type $(g,\beta)$ is identified with $$Q_{g,0}^{\uptheta}([W/\G] \times [\CC/ \CC ^*]^{k}),(\beta, 1, ..., 1) ).$$ By \eqref{rationalQ}, it is 
a DM stack,  proper over $W/_{\aff}\G$, with
a canonical perfect obstruction theory. Note that  
$$2g-2+\sum_{i=1}^k\ke_i +\beta(L_\theta)>0$$ is a necessary condition for the moduli stack to be non-empty.

In the rest of the paper we will be interested in a particular case.
Namely, 
replace $k$ by $m+k$ and then let $\ke _i = 1$ for all $i\le m$
and $\ke _{m+j} =\ke$, with $\ke$  a fixed rational number in $(0, 1]$ for $j=1, ..., k$.
We denote the ordered markings by $x_1, ..., x_m, y_1, ..., y_k$.
Hence, if $$((C, {\bf x}:=(x_1,..., x_m), {\bf y}:=( y_1, ..., y_k)), f)$$ is $(\theta, \upepsilon)$-stable, then
$(C, {\bf x})$ is a $m$-pointed prestable curve and $x_i$ are not
base points of $f$. In addition, while the points $y_j$ are allowed to coincide, no point $y_j$ may coincide with any of the $x_i$'s.
In this case, we also simply say that it is $(\theta, \ke)$-stable. 
Denote by 
\[  Q_{g, m|k}^{\theta, \ke} ([W/\G], \beta )\]
the moduli space of $(\theta, \ke)$-stable maps to $[W/\G]$ of type $(g, m|k, \beta)$.

If $((C, {\bf x}, {\bf y}), f)$ is  $(\lambda \theta _0, \ke)$-stable 
 for every sufficiently small rational number $0<\lambda<<1$ (respectively, every sufficiently large rational number $\lambda$, every $0<\ke<< 1$, 
...), then we say that it is $((0+)\cdot \theta _0, \ke)$-stable (respectively, $(\infty \cdot \theta _0, \ke )$-stable, $(\theta , 0+)$-stable,
...). Thus, from now on we consider the following extended cases
\[ (\theta , \ke ) \in  (\chi (\G )_{\QQ}  \cup \{ (0+) \cdot \theta _0, \infty \cdot \theta _0 \})  \times 
(((0, 1]\cap \QQ )\cup \{ 0+ \} ) ). \]
We treat $0+$ as an infinitesimally small {\it positive} rational number.

\begin{Rmk}
When $[W/\G] = [\CC ^{n+1}/\CC ^*]$ with $\WmodG = \PP ^n$, 
it is worth to note that the genus $1$ moduli space $Q_{1, 0|k} ^{\id, 0+} ([\CC ^{n+1}/\CC ^* ], \beta )$ is a smooth DM-stack over $\CC$ since the
obstruction vanishes (see \cite{MOP}). 
\end{Rmk}

\subsection{Evaluation maps}

There are evaluation maps at $y_j$, $j=1, ..., k$,
\[ \hat{ev}_j : Q_{g, m|k}^{\theta, \ke} ([W/\G], \beta ) \ra [W/\G] \] 
as well as the usual evaluation maps $ev_i$ at $x_i$, $i=1, ..., m$,
\[ \xymatrix{  Q_{g, m|k}^{\theta, \ke}([W/\G], \beta ) \ar[rd]_{\text{proper}} \ar[r]^{\ \ \ \ \ \ ev_i} &  \WmodtG \ar[d]^{\text{proper}}\\
                                             &      W/_{\aff}\G }  \]  compatible with canonical maps to $W/_{\aff}\G$. 
                                             The evaluation maps
$ev_i$, $i\in [m]:=\{ 1,..., m\}$ are proper, so the push-forward of homology or Chow classes 
 on $Q_{g, m|k}^{\theta, \ke}([W/\G], \beta )$ is well-defined.

\section{The big $\dsJ$-functions}

\subsection{The Novikov ring} Let an algebraic torus $\T$ act on $W$, commuting with  the $\G$-action. Recall we allow the case when
$\T$ is the trivial group. Denote $$ H^*_{\T} (\Spec (\CC), \QQ ) = \QQ [\kl_1, ..., \kl_r]$$  the $\T$-equivariant cohomology 
of a point $\Spec (\CC)$, where $r$ is the rank of $\T$.
Define the Novikov ring
\begin{equation*} \Lambda := \{ \sum _{\beta \in \mathrm{Eff}(W, \G, \theta ) } a_\beta q^\beta :  a_\beta \in \QQ \},
\end{equation*}
the $q$-adic completion of the semigroup ring $\QQ[\mathrm{Eff}(W, \G, \theta )]$, and set
\begin{align*}  &\Lambda_{\T}  :=  \Lambda\otimes_\QQ \QQ [\lambda _1, ..., \lambda _r] ,  \\
                          & \Lambda _{\T,\mathrm{loc}}  := \Lambda_\T \ot \QQ (\lambda _1, ..., \lambda _r).\end{align*}

\subsection{Weighted graph spaces}
As in \eqref{graph}, we define the $(\theta, \ke)$-{\em stable quasimap graph space} as follows:
 \[ QG_{g, m|k, \beta }^{\theta, \ke} ([W/\G] ) :=  Q_{g, m|k} ^{\theta \oplus 3\mathrm{id}, \ke} ([W\times\CC^2/ \G \times\CC^*], (\beta , 1)). \]
A $\CC$-point of the graph space is described by data 
$$((C, {\bf x}, {\bf y}), (f,\varphi):C\lra [W/\G]\times [\CC^2/\CC^*]).$$ 
Since $\ell_{3\mathrm{id}}(p)$ equals either $0$ or $3$ for every smooth point $p\in C$, stability implies that $\varphi$ is a regular map to $\PP^1=\CC^2/\!\!/_{\mathrm{id}}\CC^*$, of class $1$.
Hence the domain curve $C$ has a distinguished irreducible component $C_0$ canonically isomorphic to $\PP ^1$ via $\varphi$. 
The \lq\lq standard" $\CC ^*$-action, 
$$t\cdot [\xi _0, \xi _1] = [t\xi _0, \xi _1], \text{ for } t\in \CC ^*, [\xi _0, \xi _1]\in \PP ^1,$$
induces  a $\CC ^*$-action on the graph space. With this convention, the $\CC^*$-equivariant first Chern class of the tangent line $T_0\PP ^1$ at $0\in\PP^1$ is $c_1^{\CC ^*}(T_0\PP ^1)=z$,
where $z$ denotes the equivariant parameter, i.e., $H^*_{\CC ^*} (\Spec (\CC )) = \QQ [z]$.

There are  $\T\times\CC ^*$-equivariant evaluation morphisms
\begin{align}\nonumber &\hat{\widetilde{ev}}_j :  QG_{g, m|k, \beta }^{\theta, \ke} ([W/\G] ) \ra       [W/\G]\times \PP^1 , & j=1,\dots,k, \\ 
        \nonumber     & \widetilde{ev} _i : QG_{g, m|k, \beta }^{\theta, \ke} ([W/\G] ) \ra       \WmodtG\times \PP^1 ,  & i=1,\dots,m, \end{align}   
and
\begin{align}\nonumber &\hat{ev}_j:=pr_1\circ \hat{\widetilde{ev}}_j :  QG_{g, m|k, \beta }^{\theta, \ke} ([W/\G] ) \ra       [W/\G] , & j=1,\dots,k, \\ 
        \nonumber     &ev_i:=pr_1\circ \widetilde{ev} _i : QG_{g, m|k, \beta }^{\theta, \ke} ([W/\G] ) \ra       \WmodtG ,  & i=1,\dots,m, \end{align}  
where $pr_1$ is the projection to the first factor.

Since to give a morphism $f: C \ra [W/\G]$ amounts to giving a principal $\G$-bundle $P$ on $C$ and a section $u$ of $P\times _{\G} W$, 
there is a natural morphism $C\ra E\G \times _{\G} W$ and hence a pull-back homomorphism
 \[ f^*:  H^*_{\G}(W)  \ra H^*(C). \] 
 Now apply this to the universal curve over the moduli space, with its universal morphism to $[W/\G]$. 
 The evaluation maps are the compositions of the universal morphism with the sections of the
universal curve giving the markings and are $\T\times\CC ^*$-equivariant.
 We obtain in this way the pull-back homomorphism 
 \[ \hat{ev}_j^*: H^*_{\G\times \T}(W, \QQ)\otimes_\QQ \QQ[z] \ra H^*_{\T\times\CC^*} (QG_{g, m|k, \beta }^{\theta, \ke} ([W/\G] ) , \QQ ) \]  
 associated to the evaluation map $\hat{ev}_j$.
 
 We identify as usual $H^*_{\T} ([W/\G], \QQ):= H^*_{\G\times \T}(W, \QQ) $.

Now fix $(\theta,\ke)$ (including the cases $\theta=(0+)\cdot\theta_0$ and $\ke=0+$) and consider the graph spaces 
$QG_{0, 0|k, \beta }^{\theta, \ke} ([W/\G] )$. The description of the fixed loci for the $\CC^*$-action is parallel to the one 
given in \cite[\S4.1]{CKg0} for the unweighted case. 
In particular, we have the part $F_{k,\beta}$ of the $\CC ^*$-fixed locus for which the markings and the entire class $\beta$ are  over $0 \in \PP ^1$. It comes with
a natural {\it proper} evaluation map $ev_{\bullet}$ at the generic point of $\PP ^1$:
\[ ev_\bullet:  F_{k,\beta} \ra \WmodG .  \]
When $k\ke+\beta(L_\theta) >1$, we have the identification
\begin{equation*}F_{k,\beta}\cong Q_{0, 1|k}^{\theta, \ke} ([W/\G], \beta ),
\end{equation*}
with $\eb=ev_1$, the evaluation map at the weight $1$ marking.

On the other hand, when $k\ke+\beta(L_\theta) \leq1$, then
\begin{equation*}F_{k,\beta}\cong F_\beta\times 0^k\subset F_\beta\times (\PP^1)^k,
\end{equation*}
with $F_\beta$ the $\CC^*$-fixed locus in $QG^{(0+)\cdot\theta}_{0,0,\beta}([W/\G])$ for which the class $\beta$ is concentrated over $0\in\PP^1$.  This $F_\beta$ parametrizes
quasimaps of class $\beta$
$$f:\PP^1\lra [W/\G]$$ with a base-point of length $\beta(L_\theta)$ at $0\in\PP^1$. The restriction of $f$ to $\PP^1\setminus\{0\}$ is a constant map to $\WmodtG$ and this defines the evaluation
map $\eb$.

As in \cite{CK, CKM, CKg0}, we define the big $\dsJ$-function as the generating function for
the push-forward via $ev_\bullet$ of localization residue contributions of $F_{k,\beta}$:

\begin{Def}
 For ${\bf t}\in  H^*_{\T} ([W/\G], \QQ) 
\subset H^*_{\T} ([W/\G], \QQ )\ot _{\QQ} \QQ[z]$, let
 \begin{align*} \mathrm{Res}_{F_{k,\beta}}({\bf t}^k) &:=( \iota _\beta ^* (\prod_{i=1}^k \hat{ev}_i^*({\bf t})))\cap\mathrm{Res}_{F_{k,\beta}}[\QGraphok  ]^{\mathrm{vir}} \\
 &:=\frac{( \iota _\beta ^* (\prod_{i=1}^k \hat{ev}_i^*({\bf t})))\cap [F_{k,\beta}]^{\mathrm{vir}}}
 {\mathrm{e}^{\CC^*}(N^{\mathrm{vir}}_{F_{k,\beta}})},
 \end{align*}
 where 
$\iota _\beta : F_\beta \hookrightarrow \QGraphok$ is the inclusion, $N^{\mathrm{vir}}_{F_{k,\beta}}$ is the virtual normal bundle and $\mathrm{e}^{\CC^*}$ denotes the equivariant Euler class.

 The big $\dsJ$-function for the $(\theta,\ke)$-stability condition is
\begin{equation}\label{Je}
\dsJ^{\theta, \ke } (q,{\bf t}, z):=\sum_{\beta\in\mathrm{Eff}(W, \G, \theta) }\sum_{k\geq 0} \frac{q^\beta}{k!}
(\eb)_*\mathrm{Res}_{F_{k,\beta}}({\bf t}^k)
\end{equation} as a formal function in $\bf t$.
\end{Def}

Usually we will only be concerned with the restriction of ${\bf t}$ to a finite dimensional subspace of $H^*_{\T} ([W/\G], \QQ)$ as follows.
Let \[ \kappa : H^*_{\T} ([W/\G], \QQ) \ra H^*_{\T} (\WmodtG ,\QQ ) \] denote the 
Kirwan map (surjective, by \cite{Kir}) induced from the open immersion $\WmodtG = [W^{ss}(\theta )/\G] \subset [W/\G]$.

Fix a homogeneous basis $\{\gamma_i\}_i$ of $H^*_{\bT} (\WmodG)$ and choose homogeneous lifts $\tilde{\gamma} _i\in H^*_{\T} ([W/\G], \QQ)$ with $\kappa(\tilde{\gamma} _i)=\gamma_i$.
After restricting to
 \[ {\bf t} := \sum_i t_i  \tilde{\gamma} _i, \]  
the  big $\dsJ$-function \eqref{Je}  is a formal function in the finitely many variables $\{ t_i\}$. 

We remark that $\hat{ev}^*_i({\bf t})$ is a class in $H^*_{\T\times\CC ^*}(\QGraphok ,\QQ)$.

Since \[ QG_{0, 0|k, \beta=0 }^{\theta, \ke}  ([W/\G]  ) = \WmodtG \times (\PP ^1)^k  \supset F_{k,0} = \WmodtG \times 0^k, \]
we conclude that
 \begin{equation}\label{Jmodq}  \dsJ^{\theta , \ke} ({\bf t}, z) = e^{\kappa({\bf t}) /z} + O(q). \end{equation}

\medskip

From now on, unless otherwise stated, {\em assume that the $\T$-fixed locus $(W/_{\aff}\G )^{\bf T}$ is proper over $\CC$ (i.e., a finite set of points).}
This implies that the $\T$-fixed loci in $\WmodtG$, as well as the $\T$-fixed loci in all moduli stacks of $(\theta,\ke)$-stable quasimaps are also proper.

\medskip

\subsection{$E$-Twisting}\label{E-Twisting} Let $E$ be a finite dimensional $\T\times \G$-representation space. Then
twisting by the $\T$-equivariant vector bundle 
$$\underline{E}: = W\times _{\G} E$$ on $[W/\G]$  can be  considered, via
 replacements
 \begin{align*} 
 [Q_{0, m | k }^{\theta, \ke} ([W/\G], \beta ) ]^{vir} & \mapsto \mathrm{e}^{\T}(\pi _* f^* \underline{E} ) \cap [Q_{0, m | k }^{\theta, \ke} ([W/\G], \beta ) ]^{vir} , \\
  [QG_{0, m | k, \beta }^{\theta, \ke} ([W/\G] )]^{vir} & \mapsto \mathrm{e}^{\T}(\pi _* f^* \underline{E} ) \cap [QG_{0, m | k, \beta }^{\theta, \ke} ([W/\G] ) ]^{vir}  
  \end{align*}
as in \cite[\S 7.2.1]{CKg0}, assuming that
\begin{equation}\label{hyper-cond} R^1 \pi _* f^* \underline{E} = 0 \ \text{ for all } \beta \in \mathrm{Eff} (W, \G, \theta ). \end{equation}
Here $\pi$ is the projection from the universal curve $\cC$, $f:\cC\lra [W/\G]$ is the universal map to the quotient stack, and $\mathrm{e}^{\T}$ is the equivariant Euler class. 
 Note that if $\cP$ denotes the universal
principal $\G$-bundle on $\cC$, then $f^*\underline{E}=\cP\times_\G E$.

Now we can define $\tilde{\dsJ}^{\theta, \ke, E}$ exactly parallel to \cite[\S 7.2.1]{CKg0}:

\begin{align*}   &\tilde{\dsJ}^{\theta, \ke, E}(q,{\bf t}, z)=\left(\one     +\frac{\kappa({\bf t})}{z}\right)  \mathrm{e}^{\T}(\uE |_{\WmodG}) +
\sum_{(k,\beta)\neq (0,0),(1,0)}\frac{q^\beta}{k!}\times \\
&\times(\eb)_*\left(\iota_\beta^*(\prod_{i=1}^k \hat{ev}_i^*({\bf t}))\cap\mathrm{Res}_{F_{k,\beta}}({\mathrm{e}}^{\T}(\pi_* f^*  \uE)\cap [QG_{0, 0 | k, \beta }^{\theta, \ke} ([W/\G] ) ]^{\mathrm{vir}} ) \right).
\end{align*}

\subsection{Results}

\begin{Conj}\label{Conj1}
The function $\tilde{\dsJ } ^{\theta, \ke , E}$ is on the Lagrangian cone encoding 
the genus $0$, $\T$-equivariant, $\underline{E}|_{\WmodG}$-twisted Gromov-Witten theory of $\WmodtG$ with the Novikov ring $\Lambda_\T$ (see \cite{CG, Giv_Sym} for
the definition of the Lagrangian cone).
\end{Conj}

\begin{Thm}\label{main1}
 If the $\bf T$-action on $\WmodtG$ has only isolated fixed points and 
 only  isolated 1-dimensional orbits, 
Conjecture \ref{Conj1} holds true. 
\end{Thm}


\section{Proof of Theorem \ref{main1}}

To keep the presentation simple, we drop the $E$-twisting. However, an identical proof works in the twisted case as well.

Let $\{ \gamma _i := \kappa (\tilde{\gamma} _i) \}$ be a basis of
$$H^*_{{\bf T}, \mathrm{loc}} (\WmodtG) := H^*_{\T} (\WmodtG, \QQ)\ot _{\QQ[\lambda _1, ..., \lambda _r]} \QQ (\lambda _1, ..., \lambda _r)$$ and $\{ \gamma ^i \}$ be the dual basis 
with respect to the $\T$-equivariant Poincar\'e pairing $\lan\;, \ran $ of $\WmodtG$.

\subsection{The $\dsS$-operator}
For $\sigma _i\in H^*_{\T, \mathrm{loc}} (\WmodtG)$ and $\delta _j \in H^*_{\T} ([W/\G ], \QQ)$, denote
\[ \lan \sigma _1 \psi _1^{a_1}, ..., \sigma _m \psi _m ^{a_m} ; \delta _1, ..., \delta _k \ran ^{\theta, \ke}_{g, m|k, \beta }  
:= \int _{[Q_{g, m|k} ^{\theta, \ke} (\WmodG, \beta ) ] ^{\vir}} \prod _i ev_i ^* (\sigma _i ) \psi _i ^{a_i} \prod _j \hat{ev}_j ^* (\delta _j) ,\]
where $\psi _i$ is the psi-class associated to the $i^{\mathrm{th}}$-marking of weight $1$.
In the case $W/_{\aff}\G$ is not a single point, so that $\WmodG$ is only quasi-projective, the integral is understood as usual via the virtual localization formula.

Define for a formal ${\bf t} = \sum t_i \tilde{\gamma}_i$ in $H^*_{\T} ([W/\G], \QQ)$
\begin{align*} \lla \sigma _1 \psi _1^{a_1}, ..., \sigma _m \psi _m ^{a_m} \rra  ^{\theta, \ke} _{g, m, \beta } 
& := \sum _{k\ge 0} \frac{1}{k!} \lan \sigma _1 \psi _1^{a_1}, ..., \sigma _m \psi _m ^{a_m} ; {\bf t}, ..., {\bf t} \ran ^{\theta, \ke} _{g, m|k, \beta } ,\\
 \lla \sigma _1 \psi _1^{a_1}, ..., \sigma _m \psi _m ^{a_m} \rra  ^{\theta, \ke} _{g, m} 
 &:= \sum _\beta q^\beta \lla \sigma _1 \psi _1^{a_1}, ..., \sigma _m \psi _m ^{a_m} \rra  ^{\theta, \ke} _{g, m, \beta } .
\end{align*}

\begin{Rmk}
Let $\T$ be the trivial group. Then without the assumption that $W/_{\aff}\G$ is a point,
we may regard the above invariants as taking
values in Borel-Moore homology $H^{\mathrm{BM}}_*(W/_{\aff}\G, \Lambda _{\nov})$ using the canonical proper morphism
$ Q_{g, m|k} ^{\theta, \ke} (\WmodG, \beta ) \ra W/_{\aff}\G$.
\end{Rmk}

We define next the $\dsS$-operator: for $\gamma\in H^*_{\T,\mathrm{loc}} (\WmodtG , \Lambda)$,
\begin{equation}\label{S-op}  \dsS^{\theta, \ke} _{\bf t} (z) (\gamma ) := \sum_i \gamma _i \lla \frac{\gamma ^i}{z-\psi }, \gamma \rra ^{\theta, \ke} _{0, 2}  = \gamma + O(1/z) . \end{equation}

Let $\overline{M}_{0, 2|\ke \cdot k}$ be the Hassett moduli space of $(1,1, \ke, ..., \ke)$-weighted stable pointed curves. By \cite{H} there is a
natural birational contraction
\[ \overline{M}_{0,2+k } \ra \overline{M}_{0, 2|\ke \cdot k }. \] 
From this and the identification 
$$Q_{0, 2|k}^{\theta, \ke} (\WmodG, 0) = \overline{M}_{0, 2|\ke \cdot k} \times \WmodG $$
of the moduli spaces with class $\beta=0$, one obtains that the $\dsS$-operator has the asymptotic expansion in $q$
\begin{equation}      \label{Smodq} \dsS^{\theta, \ke} _{\bf t} (z) (\gamma )= e^{\kappa({\bf t})/z} \gamma + O(q)  . \end{equation}

Let $p_0$ and $p_\infty$ be $\CC ^*$-equivariant cohomology classes of $\PP ^1$ defined by their restriction at the fixed points:
$$p_0 |_0 = z, \; p_0|_{\infty} = 0, \; p_\infty |_{0} =0, \;p_\infty |_{\infty} =-z.$$
Consider the graph space double bracket 
\begin{align*}& \lla \sigma _1\ot p_0, \sigma _2\ot  p_\infty \rra ^{QG^{\theta, \ke}} _{0, 2} :=\\ &\sum_{k,\beta}\frac{q^\beta}{k!} \int_{[ QG_{0, 2|k, \beta }^{\theta, \ke} ([W/\G] )]^{\mathrm{vir}}}
\widetilde{ev} _1^*( \sigma _1\ot p_0)\widetilde{ev} _2^*(\sigma _2 \ot p_\infty)\prod_{j=1}^k \hat{ev}_j^*({\bf t})=\\
&\lan \sigma _1, \sigma _2 \ran + O(q). \end{align*}
Virtual $\CC ^*$-localization gives the factorization  
\begin{align*} \lla \sigma _1\ot p_0, \sigma _2\ot p_\infty \rra ^{QG^{\theta, \ke}} _{0, 2}
&=\sum _i  \lla \sigma _1, \frac{\gamma ^i}{z-\psi} \rra^{\theta, \ke}_{0,2} \lla \frac{\gamma _i}{-z-\psi }, \sigma _2 \rra^{\theta, \ke}_{0,2} \\
&= \lan \sigma _1, \sigma _2 \ran + O(1/z) .\end{align*}
On the other hand,  $\lla \sigma _1\ot p_0, \sigma _2\ot p_\infty \rra ^{QG^{\theta, \ke}} _{0, 2}$  is
well-defined without any localization with respect to $z$. Hence we conclude the following
(for details, see the proof of Proposition 5.3.1 of \cite{CKg0}).

 \begin{Prop}\label{unitary} 
The operator  $(\dsS^{\theta, \ke})^\star _{\bf t}(-z)$ defined by
\[   (\dsS^{\theta, \ke})^{\star}_{\bf t} (-z)(\gamma ) = \sum_i \gamma ^i \lla \gamma _i, \frac{\gamma}{-z-\psi}\rra ^{\theta, \ke}_{0,2}        \] 
is the inverse of $S^{\theta, \ke}_{\bf t} (z)$, i.e.,
\[ (\dsS^{\theta, \ke})^\star _{\bf t} (-z) \circ \dsS^{\theta, \ke} _{\bf t} (z) = \mathrm{Id}. \]
\end{Prop}

\subsection{The $P$-series}
For ${\bf t}=\sum_i t_i\tilde{\gamma}_i $, let
\begin{align}      P^{\theta, \ke} ({\bf t}, z) & := \sum _i \gamma ^i \lla \gamma _i\ot p_\infty \rra ^{QG ^{\theta, \ke} }_{0,1}  \nonumber \\
      \label{factor}        & =  (\dsS^{\theta, \ke})^\star _{\bf t} (-z) (\dsJ^{\theta, \ke} ({\bf t}, z) ) .
             \end{align}
             The latter equality follows from the $\CC ^*$-localization factorization. From this and Proposition \ref{unitary} we obtain the following analog of the Birkhoff factorization Theorem 5.4.1 of
             \cite{CKg0}.
\begin{Prop}\label{JSP}
 \[ \dsJ^{\theta, \ke}  ({\bf t}, z) = \dsS^{\theta, \ke} _{\bf t}(z) (P^{\theta, \ke} ({\bf t}, z)) . \]
\end{Prop}

Note that  Proposition \ref{JSP} together with \eqref{Jmodq} and \eqref{Smodq}  implies that  
 \[ P^{\theta, \ke} ({\bf t}, z) = \one + O(q). \]

\subsection{Polynomiality}

For $\mu \in (\WmodG) ^{\T}$, let $$\delta _\mu:=(\iota _{\mu })_*[\mu] \in H^*_{\T}(\WmodG, \QQ)$$ where
$\iota_{\mu}$ is the $\T$-equivariant closed immersion $\{\mu\} \hookrightarrow \WmodG$.
Let \[ \dsS^{\theta, \ke} _{\mu} (q, {\bf t}, z) := \lan \dsS^{\theta, \ke} _{\bf t} (z)(\gamma ), \delta _\mu \ran , \] 
for \[ \gamma =  \sum _{\beta} q^\beta \gamma _\beta, \ \ \gamma _\beta \in H^*_{{\T}, \text{loc}} (\WmodG )[z] .\]

\begin{Lemma}
For each fixed point $ \mu \in (\WmodG)^{\T}$, the product series
\[  \dsS ^{\theta, \ke} _\mu (q, {\bf t}, z ) \dsS ^{\theta, \ke} _\mu (qe^{-zyL_\theta}, {\bf t}, -z) \]
has no pole at $z=0$. Here $y$ is a formal variable and $(qe^{-zyL_\theta})^{\beta} := q^{\beta} e^{-zy\beta(L_\theta)}$.
\end{Lemma}

\begin{proof}
The proof is identical to the proof of Lemma 7.6.1 of \cite{CKg0}.
\end{proof}

\subsection{Comparison of $S$-operators}

It is obvious from definitions that the stability condition $(\infty\cdot\theta_0,1)$ gives the usual moduli spaces of stable maps to $\WmodG$ (or to $\WmodG\times\PP^1$ for the graph spaces), 
hence the resulting theory is the Gromov-Witten theory of $\WmodG$.  We will simply write $(\infty,1)$ for this stability condition. This is justified, since the theory
is independent on the choice of $\theta_0$, as long as we stay in the same GIT chamber for the action of $\G$ on $W$.

\begin{Conj}\label{Main_Conj} Let $(\theta,\ke)$ be arbitrary, including all asymptotic cases. Then
\begin{enumerate}
\item  \[ \dsS_{\bf t}^{\theta, \ke} (\one ) = \dsS_{\tau ({\bf t})} ^{(\infty, 1)} (\one ) \]
with
\[ \tau ({\bf t}):= \kappa ({\bf t}) + \sum _{\beta \ne 0} q^\beta \sum_i \gamma _i \lla \gamma ^i , \one \rra _{0, 2, \beta}^{\theta, \ke}. \] 

\item  For ${\bf t}:=\sum_i t_i \tilde{\gamma}_i$, there are unique 
\begin{align*} P^{(\infty ,1), \theta, \ke}({\bf t}, z)=\one+O(q) & \in H^*_{\bT, \mathrm{loc}}(\WmodG)[z][[q, t_j]], \\
\tau ^{(\infty, 1), \theta, \ke}({\bf t})= \kappa ({\bf t}) +O(q) & \in H^*_{\bT, \mathrm{loc}}(\WmodG)[[q, t_j]] \end{align*} such that
\begin{equation} \label{mirror formula}
\dsS_{\bf t} ^{\theta, \ke}(z) (P^{\theta, \ke} ({\bf t}, z)) = \dsS_{\tau ^{(\infty, 1), \theta, \ke} ({\bf t})} ^{(\infty,1)} (z) (P^{(\infty , 1),\theta, \ke} ( \tau ^{(\infty, 1), \theta, \ke} ({\bf t}) , z ) ). \end{equation}
\end{enumerate}
\end{Conj}

Just as in \cite[Lemma 6.4.1]{CKg0}, one can recursively construct uniquely determined series $P^{(\infty ,1), \theta, \ke}({\bf t}, z)$ and $\tau ^{(\infty, 1), \theta, \ke}({\bf t})$ with the required 
$q$-asymptotics, and which satisfy equation \eqref{mirror formula} {\it modulo} $1/z^2$. The content of part $(2)$ 
 of Conjecture \ref{Main_Conj} is that equality modulo $1/z^2$ suffices to force the equality to all orders in $1/z$.
Note that when 
 combined with Proposition \ref{JSP}, part $(2)$ implies Conjecture \ref{Conj1}.

\begin{Thm}\label{Comp_Thm}
Suppose that the induced $\T$-action on $\WmodG$ has only isolated $\T$-fixed points. Then Conjecture \ref{Main_Conj} (1) holds true.

Further, if in addition $\WmodG$ has only isolated 1-dimensional $\T$-orbits, then
Conjecture \ref{Main_Conj} (2) holds true.
\end{Thm}
\begin{proof}
The proof of the first statement  is identical with the proof of Theorem 7.3.1 of \cite{CKg0},
while the proof of the second statement is identical with the proof of Theorem 7.3.4 of \cite{CKg0}.
\end{proof}

Now the proof of Theorem \ref{main1} follows from Proposition \ref{JSP} and Theorem \ref{Comp_Thm}.

\subsection{Non-equivariant limit} If $\WmodG$ is projective then one can work with the {\it non-localized} equivariant cohomology ring $H^*_{\T}(\WmodG,\QQ)$, the Poincar\'e pairing with
values in $\QQ[\lambda_1,\dots ,\lambda_r]$, and the Novikov ring $\Lambda_{\T}$. The objects $\dsJ^{\theta, \ke}$, $\dsS_{\bf t}^{\theta, \ke}$, $P^{\theta, \ke}$, $\tau ({\bf t})$,
$P^{(\infty ,1), \theta, \ke} $, and $\tau ^{(\infty, 1), \theta, \ke}$ reduce to their non-equivariant counterparts upon setting $\lambda_1=\dots =\lambda_r=0$.


\section{Explicit Formula for the fully asymptotic stability condition}

\subsection{$\dsI$-function} Other than the Gromov-Witten chamber $(\theta,\ke) =(\infty, 1)$, the most interesting case from a computational viewpoint is 
the opposite asymptotic case $(\theta,\ke) =(0+, 0+)$ (again, the theory is independent on the choice of character in a given GIT chamber, so we drop $\theta_0$ from the notation). 
The main reason is that $QG_{0,0|k, \beta}^{0+ ,0+}([W/\G]  )$ 
is isomorphic to  \[ QG_{0,0,\beta}^{0+ ,0+}([W/\G]) \times (\PP ^1)^k. \]
The space $QG_{0,0,\beta}^{0+ ,0+}([W/\G]) $ coincides
with  $\mathrm{Qmap}_{0, 0}(\WmodtG , \beta; \PP^1)$ defined in \cite[\S 7.2]{CKM}
and was denoted by $QG_{0,0, \beta}(\WmodtG )$ in \cite[\S 2.6]{CKg0}. Further, as we already noted earlier
\[ F_{k,\beta}=F_{\beta}\times 0^k,\]
where $F_\beta=F_{0,\beta}$ is the distinguished $\CC^*$-fixed locus in $QG_{0,0,\beta}^{0+ ,0+}([W/\G]) $. 

Denote \begin{equation}\dsI =\dsI _{\WmodtG}(q,{\bf t},z) := \dsJ^{0+ , 0+}(q,{\bf t},z).\end{equation}
In this paper we will call $\dsI _{\WmodtG}(q,{\bf t},z)$ {\it the big $\dsI$-function of} $\WmodtG$. This differs from the terminology in \cite{CKM,CKg0}.
The specialization
$$\dsI _{sm} (q, z):= \dsI (q, 0, z)=\sum _{\beta} q^\beta \dsI _{\beta} (z)$$
is called the small $\dsI$-function of $\WmodtG$ (this terminology {\it does} agree with the one in \cite{CKg0, CKg}).
We have
\begin{equation}\label{I-res} \dsI _{\beta} (z) =     (ev _\bullet )_* \mathrm{Res}_{F_\beta} [QG_{0,0,\beta}(\WmodtG )] ^{\vir}. \end{equation}
As is well-known, these push-forwards of residues can often be 
explicitly calculated in closed form. For example, the case of toric varieties goes back to Givental, \cite{Givental}, see also \cite[\S7.2]{CK} for an exposition. Type $A$ flag varieties, which are
examples with non-abelian $\G$, are treated in \cite{BCK1, BCK2}. For more on the non-abelian case, see the forthcoming note \cite{CKabelianization}.

The goal of this section is to find explicit formulas for the big $\dsI$-functions for some $(W, \G, \theta )$.
To emphasize the role of class $\beta$, we write
$$\hat{ev}_{\beta}=\hat{ev}_j : F_\beta \ra [W/\G] .$$ Note that these evaluation maps do not depend on the choice of $j$ since 
all marked points  are concentrated on $0\in \PP ^1$.
It follows that
$$\dsI(q,{\bf t},z ) = \sum _\beta q^\beta (ev_{\bullet})_*( \exp (\hat{ev}^*_\beta ({\bf t})/z  ) \cap \mathrm{Res}_{F_\beta} [QG_{0,0,\beta}(\WmodtG )] ^{\vir}  ).$$

Suppose that for some $\gamma _{i,\beta} (z) \in H^*(\WmodtG)\ot \QQ[z]$,
\begin{equation}\nonumber  (ev_\bullet)^*\gamma _{i,\beta}(z) =(\hat{ev}_{\beta} ^* (\tilde{\kg} _i)) .  \end{equation} 

Then by the projection formula, the big $\dsI$-function becomes
\begin{equation}\label{big-explicit}  \sum _{\beta} e^{\sum _i t_i \gamma _{i, \beta} (z) / z}    q^{\beta} \dsI_\beta ( z)        \ \ 
\text{ for }   \ \   {\bf t} = \sum t_i \tilde{\gamma }_i .\end{equation}

Whenever the small $I$-function is known, to obtain an explicit formula for $\dsI$ it remains to find explicitly such classes $\gamma _{i,\beta} (z)$.

\begin{Rmk} By Theorem \ref{main1}, the big $\dsI$-function \eqref{big-explicit} is on the Lagrangian cone of the Gromov-Witten theory of $\WmodtG$ whenever the $\T$ action has isolated fixed points
and isolated $1$-dimensional orbits. This statement is presumably related to Woodward's result in \cite[Theorem 1.6]{Woodward}.
\end{Rmk}

\subsection{Description of $ev$}

Let $A(W)$ be the affine coordinate ring of $W$ and let $\zeta_0, \zeta_1$ be 
the homogeneous coordinates  of $\PP ^1$ defining $0\in \PP ^1$ by the equation $\zeta_0=0$.

For a sufficiently large and divisible integer $m$, the character $m\theta$ defines a morphism 
\[ \iota: [W/\G] \ra [\CC ^{N+1} / \CC ^*]_{A(W)^\G} := [\Spec (A(W)^{\G}) \times \CC ^{N+1} / \CC ^*] \]  whose restriction 
$\WmodG \ra \PP ^N_{A(W)^{\G}} $
is an embedding (see \cite[\S 3.1]{CKg0}). Let $d:=\beta (L_{m\theta})$. Recall that
$$QG_{0,0,d}(\CC ^{N+1} /\!\!/_{\mathrm{id}} \CC ^* ) = \PP (\Sym ^d ((\CC ^2)^\vee)\ot \CC ^{N+1}) ),$$
and that its $\CC ^*$-fixed distinguished part $F_d$ is $\PP ((\CC\cdot \zeta_0^d)\ot \CC ^{N+1})=\PP ^N$ (see \cite{Givental-equiv}).

Consider now the following natural diagram
\[ \xymatrix{ 
          &    F_\beta \ar'[d][dd]^{\iota}   \ar[rd]^{ev_\bullet}     &      \\
          \PP ^1 \times F_\beta \ar[rr]_{\ \ \ \ \ \ \ \ ev} \ar[dd]^{\iota}    \ar[ur]^{\pi}  &  & [W/\G] \ar[dd]^{\iota}\\
                       & \PP ^N_{A(W)^{\G}}  \ar[dr]^{\uev_\bullet} &       \\
                       \PP ^1 \times \PP ^N _{A(W) ^{\G}} \ar[rr]_{\uev} \ar[ur]^{\upi} &    &  [\CC ^{N+1} / \CC ^*]_{A(W)^\G} 
                       } \] 
where 
\begin{itemize}
\item the vertical morphisms are induced from $\iota$ (abusing notation, we denote all of them also by $\iota$);
\item  $ev, \underline{ev}$ are the universal evaluation maps; 
\item $ev_\bullet, \underline{ev}_\bullet$ are evaluation maps
at the generic point of $\PP ^1$; 
\item $\pi, \underline{\pi}$ are projections. 
\end{itemize}
All side square faces are commutative but the upper and the lower triangle faces need not be commutative. 

Let $w_0, ..., w_N$ be the homogeneous coordinates of $\PP ^N$. On the stack quotient $[\CC ^{N+1} / \CC ^*]_{A(W)^\G}$
we have  the invertible sheaf  $\mathcal{O}_{[\CC ^{N+1} / \CC ^*]_{A(W)^\G}}(1)$ attached to
the character $\mathrm{id}$. Let
$\CC _{nz}$ denote the $\CC^*$-representation space given by the character $nz=n\mathrm{id}$.

The map $\underline{ev}$ is defined by the line bundle $\mathcal{O}_{\PP ^1}(d)\boxtimes \mathcal{O}_{\PP ^N_{A(W)^\G}} (1)$
together with sections $\zeta_0^d\boxtimes w_i$, $i=0, ..., N$. Therefore as $\CC^*$-equivariant coherent sheaves
\begin{align} \nonumber  \underline{ev}^*(\mathcal{O} _{ [\CC ^{N+1} / \CC ^*]_{A(W)^\G}} (1))  & = \mathcal{O}_{\PP ^1}(d)\boxtimes \mathcal{O}_{\PP ^N_{A(W)^\G}} (1) \\
 &=  \mathcal{O}_{\PP ^1}(d)\boxtimes  \underline{ev} _\bullet ^*\mathcal{O} _{\PP ^N_{A(W)^\G}}  (1), \label{non-com}
\end{align} 
where  $\mathcal{O}_{\PP ^1} (d) |_0 = \CC _{dz}$ and $ \underline{ev} _\bullet ^*\mathcal{O} _{\PP ^N_{A(W)^\G}}  (1)$
has the trivial $\CC^*$-equivariant structure.

\begin{Lemma}\label{some-lemma} The following equality holds in $\Pic _{\CC ^*}(F_\beta)_{\QQ}$:
\[ \hat{ev}_\beta^* (L_\theta) = ev_\bullet^* (L_\theta ) \boxtimes \CC_{\beta(L_\theta)z}  ,\]
where the $\CC^*$-action on $ev_\bullet^* (L_\theta )$ is trivial.
\end{Lemma}

\begin{proof}  We take  $\iota ^*$ on \eqref{non-com} and use $\iota ^* \mathcal{O} _{\PP ^N_{A(W)^\G}}  (1) = L^{\ot m}_\theta $ to conclude the proof.
\end{proof}

\begin{Rmk}\label{some-remark}
Let $\theta'$ be another character in the same GIT chamber as $\theta$. Since the moduli spaces of weighted stable quasimaps for 
the $((0+)\cdot\theta,0+)$ and $((0+)\cdot\theta',0+)$ stability conditions also coincide, we conclude that Lemma \ref{some-lemma} also applies to $L_{\theta'}$. If the GIT chamber 
has dimension equal to the rank of the group of rational characters $\chi(\G)\otimes\QQ$, then it contains a basis of $\chi(\G)\otimes\QQ$ and therefore Lemma \ref{some-lemma}
holds for any character of $\G$ up to torsion.
\end{Rmk}

\subsection{Examples} If $W$ is a vector space and $\G\cong(\CC^*)^s$ is a torus, so that $\WmodtG$ is a nonsingular toric variety, then $H^*([W/\G])$ is a polynomial algebra over $\QQ$,
with generators $c_1(L_{\eta_i})$
corresponding to a $\QQ$-basis $\{\eta_1,\dots, \eta_s\}$ of $\chi(\G)\otimes\QQ$. By Remark \ref{some-remark}, for any polynomial $p(c_1(L_{\eta_1}),\dots,c_1(L_{\eta_s}))$ we have
$$ \hat{ev}_\beta^*p(c_1(L_{\eta_1}),\dots,c_1(L_{\eta_s}))=ev_\bullet^*p(c_1(L_{\eta_1})+\beta(L_{\eta_1})z,\dots,c_1(L_{\eta_s})+\beta(L_{\eta_s})z).$$
In particular, the classes $\gamma_{i,\beta}(z)$, and therefore the big $\dsI$-functions, are explicitly known for toric varieties. By considering twisted theories, the same is true for complete
intersections in toric varieties as well.
We exemplify with the case of projective spaces.

Let $H$ denote the hyperplane class of $\PP ^n = \CC ^{n+1}/\!\!/\CC ^*$. In this case,
applying Lemma \ref{some-lemma} to \eqref{big-explicit} and using the formula for its small $\dsI$-function from \cite{Givental-equiv}, we obtain
\[ \dsI _{\CC ^{n+1}/\!\!/\CC ^*} (q,{\bf t},z) = \sum _{d=0}^{\infty} q^d \frac{\exp (\sum _{i=0}^nt_i (H+dz)^i/z) }{\prod _{k=1}^d (H+kz) ^{n+1}}.  \] 
By the non-equivariant specialization of Theorem \ref{main1}, $\dsI _{\CC ^{n+1}/\!\!/\CC ^*} (q,{\bf t},z)$ is on the Lagrangian cone of the Gromov-Witten theory of $\PP ^n$.

More generally, let $E=\CC$ with weight a positive integer $l$ be the twisting factor, so
that $\uE |_{\PP ^n} = \mathcal{O}_{\PP ^n}(l)$. With this setting, 
\begin{align}\label{I_Hyper_P} 
\dsI  ^{E} _{\CC ^{n+1}/\!\!/\CC ^*}({\bf t}) = \sum _{d=0}^{\infty} q^d \frac{\exp (\sum _{i=0}^nt_i (H+dz)^i/z) }{\prod _{k=1}^d (H+kz) ^{n+1}}  \prod _{k=0}^{ld}  (lH + kz) . \end{align}
By Theorem \ref{main1}, the Gromov-Witten $E$-twisted $J$-function of $\PP ^n$ is
related with $\dsI  ^{E} _{\CC ^{n+1}/\!\!/\CC ^*}$ via a Birkhoff factorization (see \cite{CG} for the Birkhoff factorization procedure). Recall that the $E$-twisted $J$-function is essentially the
usual $J$-function of a hypersurface of degree $l$ in $\PP^n$.

Note that RHS of \eqref{I_Hyper_P} can be expressed as
 \[  \left(\exp (\sum _{i=0}^n \frac{t_i}{z} (z q\frac{\partial}{\partial q} + H)^i  )\right)  \dsI  ^{E} _{\CC ^{n+1}/\!\!/\CC ^*}(q, 0, z) .  \] 
 This latter expression is already considered as a special case by Iritani in  \cite[Example 4.14]{D_Module_Iritani} for 
 a reconstruction of  quantum D-modules.

\begin{Rmk}
Recent work by Coates, Corti, Iritani, and Tseng in \cite{CCIT1, CCIT2} introduces the so-called
$S$-extended $I$-function of a toric DM stack $\mathcal{X}=[(\CC^N)^{ss}(\theta)/(\CC^*)^r]$ and proves that it lies on the Lagrangian cone of the 
Gromov-Witten theory of $\mathcal{X}$. In examples, see \cite{CCIT2}, by choosing the extending set $S$ carefully, one can extract
sufficient information from the $S$-extended $I$-function to recover the big $J$-function of $\mathcal{X}$.

From the perspective of our paper (generalized to orbifold GIT targets in \cite{CCK}), the $S$-extension amounts to changing the GIT presentation of the toric
target  $\mathcal{X}$ as $[(\CC^{N+|S|})^{ss}(\theta')/(\CC^*)^{r+|S|}]$, and the $S$-extended $I$-function of \cite{CCIT1, CCIT2} coincides with the big $\mathds{I}$-function of ours (corresponding
to the {\it new} GIT presentation) 
{\em restricted} to ${\bf t} = \sum t_i \tilde{\gamma} _i$ with $\gamma _i \in H^{\le 2}(\mathcal{X})$. The additional parameters of the $S$-extended $I$-function of \cite{CCIT1} are identified with the additional ``ghost" Novikov variables (see \cite[\S5.9.2]{CKg}) of the quasimap theory for the new GIT presentation. 

Put it differently, the $S$-extended $I$-function of \cite{CCIT1} is exactly {\it Givental's  small $I$-function} for the quasimap theory of $(\CC^{N+|S|}, (\CC^*)^{r+|S|},\theta')$, as defined e.g., in 
equation $(7.3.2)$ of \cite{CKM} for the manifold case.
\end{Rmk}



\begin{thebibliography}{99}


\bibitem{AG}  V. Alexeev and M. Guy, Moduli of weighted stable maps and their gravitational descendants, J. Inst. Math. Jussieu {\bf 7} (2008), no. 3, 425--456.

\bibitem{BM} A. Bayer and Yu. I. Manin, Stability Conditions, Wall-crossing and weighted
Gromov-Witten Invariants, Mosc. Math. J. {\bf 9} (2009), no. 1, 3--32.



\bibitem{BCK1} A. Bertram, I. Ciocan-Fontanine, and B. Kim, Two proofs of a conjecture of Hori and Vafa, Duke Math. J. {\bf 126} (2005),
no. 1, 101--136.


\bibitem{BCK2} A. Bertram, I. Ciocan-Fontanine, and B. Kim,  Gromov-Witten invariants for abelian and nonabelian quotients,  J. Algebraic Geom.  {\bf 17}(2) (2008), 275--294.



\bibitem{CCK} D. Cheong, I. Ciocan-Fontanine and B. Kim, Oribifold quasimap theory, Math. Ann. 363 (2015), no. 3-4, 777-816.

\bibitem{CK} I. Ciocan-Fontanine and B. Kim, Moduli stacks of stable toric quasimaps,  Adv. in Math. {\bf 225} (2010), 3022--3051.

\bibitem {CKg0}  I. Ciocan-Fontanine and B. Kim,
Wall-crossing in genus zero quasimap theory and mirror maps,  Algebraic Geometry 4 (2014), 400-448.

\bibitem {CKg}  I. Ciocan-Fontanine and B. Kim, 
Higher genus quasimap wall-crossing for semi-positive targets, arXiv:1308.6377.

\bibitem {CKabelianization}  I. Ciocan-Fontanine and B. Kim,  The abelian/non-abelian correspondence for quasimaps, in preparation.

\bibitem{CKM} I. Ciocan-Fontanine, B. Kim, and D. Maulik,  Stable quasimaps to GIT quotients,  J. Geom. Phys. {\bf 75} (2014), 17--47.




\bibitem{CCIT1}  T. Coates, A. Corti, H. Iritani, and H.-H. Tseng, A mirror theorem for toric stacks,  arXiv:1310.4163.


\bibitem{CCIT2} T. Coates, A. Corti, H. Iritani, and H.-H. Tseng, Some applications of the mirror theorem for toric stacks, arXiv:1401.2611. 



\bibitem{CG} T. Coates and A. Givental, Quantum Riemann-Roch, Lefschetz, and Serre, Ann. Math. (2) {\bf 165}(1), (2007), 15--53.




\bibitem{Givental-equiv} A. Givental,  Equivariant Gromov-Witten invariants,  Internat. Math. Res. Notices (1996), no. 13, 613--663.



\bibitem{Givental} A. Givental, A mirror theorem for toric complete intersections, in \lq\lq Topological field theory, primitive
forms and related topics (Kyoto, 1996)", Progr. Math., 160, Birkh\" auser Boston, Boston, MA, 1998, 141--175.



\bibitem{Giv_Sym} A. Givental, Symplectic geometry of Frobenius structures, Frobenius manifolds, 91--112, Aspects Math., E36, Friedr. Vieweg, Wiesbaden, 2004.





\bibitem{H} B. Hassett, Moduli spaces of weighted pointed stable curves,
Adv. Math. {\bf 173} (2003), no. 2, 316--352.



\bibitem{D_Module_Iritani} H. Iritani, Quantum D-modules and generalized mirror transformations. Topology {\bf 47} (2008), no. 4, 225--276.


\bibitem{Janda} F. Janda, Tautological relations in moduli spaces of weighted pointed curves, arXiv:1306.6580.


\bibitem{Jin-Shi} M. Jinzenji and M. Shimizu, Multi-point virtual structure constants and mirror computation of $CP^2$-model, arXiv:1305.0999.

\bibitem{Kir} F.C. Kirwan, Cohomology of quotients in complex and algebraic geometry, Mathematical Notes 31, 
              Princeton University Press, Princeton N. J., 1984.





\bibitem{MOP} A. Marian, D. Oprea, and R. Pandharipande, The moduli space of stable quotients, Geometry \& Topology, {\bf 15} (2011), 1651--1706.




\bibitem{Woodward} C. Woodward, Quantum Kirwan morphism and Gromov-Witten invariants of quotients, arXiv:1204.1765.

\end{thebibliography}
\end{document}